\documentclass[12pt,a4paper]{amsart}

\usepackage{amsmath,amsthm,amssymb,latexsym,a4wide,tikz,multicol,tikz-cd}
\usepackage{arydshln,multirow}
\usepackage{tikz-qtree,tikz-qtree-compat}
\usepackage{mathtools,stmaryrd}
\usepackage{comment}
\usepackage{tikz}
\tikzset{font=\small}
\usepackage{enumerate}
\usetikzlibrary{matrix,arrows}
\usetikzlibrary{positioning}
\usetikzlibrary{cd}

\newtheorem{theorem}{Theorem} [section]
\newtheorem{lemma}[theorem]{Lemma}
\newtheorem{corollary}[theorem]{Corollary}
\newtheorem{proposition}[theorem]{Proposition}

\newtheorem{example}[theorem]{Example}
\newtheorem{remark}[theorem]{Remark}
\theoremstyle{definition}
\newtheorem{definition}[theorem]{Definition}

\newcommand{\dom}{\mathrm{dom}}
\newcommand{\FACT}{\mathsf{FGA}}
\newcommand{\FSPACT}{\mathsf{FSPA}}


\usepackage[active]{srcltx}

\usepackage{enumerate}
\numberwithin{equation}{section}

\title{Globalization of partial actions of semigroups}

\author{Ganna Kudryavtseva}
\address{G. Kudryavtseva: University of Ljubljana,
Faculty of Mathematics and Physics, Jadranska ulica 19, SI-1000 Ljubljana, Slovenia / Institute of Mathematics, Physics and Mechanics, Jadranska ulica 19, SI-1000 Ljubljana, Slovenia}
\email{ganna.kudryavtseva\symbol{64}fmf.uni-lj.si}
\author{Valdis Laan}
\address{V. Laan: Institute of Mathematics and Statistics, University of Tartu,
51009 Tartu, Estonia}
\email{valdis.laan\symbol{64}ut.ee}

\thanks{This work was supported by the Slovenian Research Agency grant BI-EE/20-22-010. Research of the first named author was supported by  the Slovenian Research Agency grant P1-0288. Research of the second named author was supported by the Estonian Research Council grant PRG1204.}

\begin{document}

\begin{abstract}  
We propose two universal constructions of globalization of a partial action of a semigroup on a set, satisfying certain conditions which arise in Morita theory of semigroups. One of the constructions is based on the tensor product of a partial semigroup act with the semigroup and generalizes the globalization construction of strong partial actions of monoids due to Megrelishvili and Schr\"oder. It produces the initial object in an appropriate category of globalizations of a given partial action. The other construction involves ${\mathrm{Hom}}$-sets and is novel even in the monoid setting. It produces the terminal object in an appropriate category of globalizations. While in the group case the results of the two constructions are isomorphic, they can be far different in the monoid case.
\end{abstract}

\maketitle

\section{Introduction}
The study of partial actions and related concepts is an active research area, see the survey article \cite{D19} and, e.g., \cite{ABV19,BEGRR22,DKhS21,FMS21,HV20,R22,SV22}  for more recent developments.

Strong partial monoid actions were introduced by Megrelishvili and Schr\"oder \cite[Definition 2.3]{MS04}\footnote{In \cite[Definition 2.3]{MS04} strong partial monoid actions were called  partial monoid actions.}, who have shown that they generalize partial group actions of Exel \cite{Exel98} (see also Abadie \cite{A03} and Kellendonk and Lawson \cite{KL04}) and that every strong partial monoid action can be globalized (or has an enveloping action, in the terminology of \cite{A03}), that is, it arises as a restriction of a global action. A more general notion of a partial monoid action by Hollings \cite{H07}  has found recent applications in the theory of two-sided restriction semigroups \cite{CG12,K15,K19}. Partial actions of inverse semigroups have been studied in \cite{GH09,Kh19}, of left restriction semigroups in \cite{GH09} and of two-sided restriction semigroups in~\cite{DKhK21}.

If a monoid $S$ acts partially from the right on a set $A$, we say that $A$ is a right partial $S$-act. The globalization of a strong right partial $S$-act $A$ from \cite{H07, MS04}, denote it by $A\otimes S$, is universal in the sense that for any other globalization $B$ of $A$ there is a unique morphism from $A\otimes S$ to $B$, which means that $A\otimes S$ is an initial object in the category of all globalizations of $A$. In the case where $S$ is a group, it follows from \cite[Theorem 5.4]{KL04} that $A\otimes S$ is, up to isomorphism, the only $A$-generated globalization of $A$. This raises the question if the same holds also in the monoid case, and if the answer is negative, whether  the category of all $A$-generated globalizations of $A$ possesses a terminal object, which could be regarded as the 'smallest' $A$-generated globalization of $A$, while $A\otimes S$ being the freest such globalization. 

In this note, we propose a new construction of a globalization of $A$ which involves ${\mathrm{Hom}}$-sets, denoted $A^S$, and prove that it produces the terminal object in the category of all $A$-generated globalizations of $A$. While in the group case $A\otimes S$ and $A^S$ are isomorphic, they can be far different with infinitely many intermediate non-isomorphic objects between them in the monoid case (see Example \ref{ex:example}). We hope that the universal globalization $A^S$ will find applications in the study of various globalization problems for partial actions of semigroups, monoids and beyond.

We work at the level of generality of semigroups and
generalize to partial acts two classes  of global semigroup acts, called  firm and nonsingular, which arise in the Morita theory of semigroups  (see, e.g., \cite{CS,LMR18,Lawson2011}). 
We construct globalizations of strong partial acts in these classes using the tensor product and the ${\mathrm{Hom}}$-set constructions, prove the universal properties for the constructed globalizations and derive respective corollaries for the special case of strong partial monoid actions.

The structure of the paper is as follows. In Section \ref{s:actions} we collect necessary definitions and basic facts about partial semigroup actions and their globalizations. In Section \ref{s:tensor} for a semigroup $S$ we construct the tensor product globalization $A\otimes S$ of a strong and firm partial $S$-act $A$ and prove its universal property (see Theorem \ref{th:tensor}). We then prove that the category of firm global $S$-acts is a reflective subcategory of the category of firm and strong partial $S$-acts with the tensoring globalization functor being the reflector (see Theorem~\ref{th:reflective}). In Section \ref{s:hom} we construct the ${\mathrm{Hom}}$-set globalization $A^S$ of a unitary, strong and nonsingular partial $S$-act $A$ and prove its universal property (see~Theorem \ref{th:hom}). Finally, in Section \ref{s:monoids} we apply our results to strong partial monoid actions, see Theorem~\ref{th:monoid}, Corollary \ref{cor:groups} and Example \ref{ex:example}. 

Throughout the text, the notation $S$ is used for a semigroup and $A,B,C$ for sets or for (global or partial) right $S$-acts.

\section{Actions and partial actions of semigroups}\label{s:actions}

\subsection{Partial actions of semigroups}
We first recall the definition of an action of a semigroup on a set. We will refer to  actions  as  
global actions, to emphasise their difference from partial actions, which we consider along with global actions throughout the paper.
\begin{definition} (Global action)
 A map $*: A\times S \to A, \;\; (a,s)\mapsto a* s$,
is called a {\em right global action} of $S$ on $A$ if for all $a\in A$ and $s,t\in S$ we have $(a* s)* t = a* st$. 
\end{definition}

If $X,Y$ are sets then by a {\em partial map} $f\colon X\to Y$ we understand a map $Z\to Y$ where $Z\subseteq X$. We write $Z={\mathrm{dom}}(f)$. If $z\in {\mathrm{dom}}(f)$ we say that $f(z)$ is {\em defined}.

\begin{definition}(Partial action) \label{def:partial} A partial map $\cdot: A\times S \to A, \; (a,s)\mapsto a\cdot s$, is called a {\em right partial action} of $S$ on $A$ if the following condition holds:
\begin{enumerate}
    \item[(PA)] If $a\cdot s$ and $(a\cdot s)\cdot t$ are defined then $a\cdot st$ is defined and $(a\cdot s)\cdot t = a\cdot st$.
\end{enumerate}
\end{definition}

Left global actions and left partial actions can be defined similarly. In this note, unless explicitly stated otherwise, we deal with right global and partial actions. Thus by a partial action (or by a global action) we mean  a right partial action (or a right global action).

The notion of a partial action generalizes that of a global action. If $\cdot$ is a partial action of $S$ on $A$, we say that $(A,\cdot)$ is a {\em partial} $S$-{\em act}, and when $\cdot$ is clear from the context, we  sometimes write $A$ for $(A, \cdot)$. In the case where $\cdot$ is a global action on a set $B$ we say that $(B,\cdot)$ is a {\em global} $S$-{\em act}. If $(B,\cdot)$ is a global $S$-act and $C\subseteq B$ is such that $c\cdot s\in C$ for all $c\in C$ and $s\in S$ we say that $C$ is a {\em subact of} $B$.

 \begin{definition}\label{def:su} (Unitary and strong partial actions) A right partial 
 action $\cdot$ of $S$ on $A$ is called {\em strong} if for all $a\in A$ and $s,t\in S$ the following condition holds:
\begin{enumerate}
    \item[(S)] If $a\cdot s$ and $a\cdot st$ are defined then $(a\cdot s)\cdot t$ is defined and $(a\cdot s)\cdot t = a\cdot st$. 
\end{enumerate}
It is called {\em unitary} if the following condition holds:
\begin{enumerate}
    \item[(U)] For each $a\in A$ there are $b\in A$ and $s\in S$ such that $b\cdot s$ is defined and $a=b\cdot s$.
\end{enumerate}
\end{definition} 

In view of (PA), condition (S) is equivalent to the condition that under the assumption that $a\cdot s$ is defined, we have that $a\cdot st$ is defined if and only if $(a\cdot s)\cdot t$ is defined in which case $(a\cdot s)\cdot t = a\cdot st$.

Note that in the definition of a partial action of a monoid due to Hollings \cite[Definition~2.2]{H07} it is required that $\cdot$ satisfies (PA) and the following condition:
\begin{enumerate}
    \item[(Um)] For all $a\in A$: $a\cdot 1$ is defined and $a\cdot 1=a$.
\end{enumerate}

The following lemma shows that our notion of a strong and unitary partial action of a semigroup generalizes the notion of a strong partial action of a monoid, see Hollings~\cite{H07}.

\begin{lemma}\label{lem:monoid}
If $S$ is a monoid, then $\cdot$ is a strong and unitary partial action of $S$ on a set $A$ (that is, conditions (PA),(U) and (S) hold) if and only if $\cdot$ is a strong partial action of $S$ on $A$ in the sense of \cite{H07} (that is, conditions (PA), (Um) and (S) hold).
\end{lemma}

\begin{proof} Suppose that $\cdot$ satisfies conditions (PA), (U) and (S). We show that (Um) is satisfied. Let $a\in A$ and write it as $a=b\cdot s$, where $b\in A$, $s\in S$. Since $b\cdot s$ and $b\cdot (s1)$ are defined, condition (S) implies that $(b\cdot s)\cdot 1$ is defined and $(b\cdot s)\cdot 1 = b\cdot (s1)$. Then $a\cdot 1$ is defined and $a\cdot 1=a$, as required. The reverse implication is clear.
\end{proof}

\begin{definition} (Restriction of an action)
Let $*$ be a right global action of $S$ on a set $B$ and $A\subseteq B$ a subset. The {\em restriction} of $*$ to $A$ is the partial map $\cdot \colon A\times S\to A$ where for $a\in A$ and $s\in S$  the element $a\cdot s$ is defined if and only if $a*s \in A$ in which case $a\cdot s= a*s$.  
\end{definition}

\begin{lemma}\label{lem:strong} A restriction of a global action of $S$ is a strong partial action of $S$. 
\end{lemma}

\begin{proof}
Let $*$ be a global action of $S$ on a set $B$ and $A\subseteq B$. Let $\cdot$ be the restriction of $*$ to $A$. Suppose that $a\cdot s$ and $a\cdot st$ are defined. Then $a\cdot s = a*s\in A$ and $a\cdot st = a*st \in A$. It follows that $(a\cdot s)\cdot t$ is defined and equals $(a*s)*t=a*st=a\cdot st$. 
\end{proof}

\begin{definition}\label{def:morphism} (Morphisms of partial acts) Let $\cdot$ and $\circ$ be partial actions of $S$ on the sets $A$ and $B$, respectively, and let $\varphi\colon A\to B$ be a map. We say that $\varphi$ is a {\em morphism} (or a {\em homomorphism}) of partial $S$-{\em acts from} $A$ {\em to} $B$ provided that for all $a\in A$ and $s\in S$ if $a\cdot s$ is defined then $\varphi(a)\circ s$ is defined and $\varphi(a\cdot s) = \varphi(a)\circ s$. 
\end{definition}

The set of all morphisms from $(A,\cdot)$ to $(B,\circ)$ will be denoted by ${\mathrm{Hom}}(A,B)$.

A morphism of global acts \cite[Definition I.4.15]{KKM} is a special case of Definition  \ref{def:morphism}.  

\subsection{Globalization of a partial action}

\begin{definition}\label{def:glob} (Globalization)
Let $\cdot$ be a right partial action  of a semigroup $S$ on a set $A$. A {\em globalization} of $\cdot$ is a right action $*$ of $S$ on a set $B$ and an injective map $\iota: A\to B$ such that the following conditions hold:
\begin{enumerate}
\item[(G1)] For all $a\in A$ and $s\in S$: $a\cdot s$ is defined if and only if $\iota(a)* s \in \iota(A)$. 
\item[(G2)] For all $a\in A$ and $s\in S$: if $a\cdot s$ is defined  then $\iota(a\cdot s) = \iota(a)*s$. 
\end{enumerate}
We also say that the right global $S$-act $(B,*)$ is a {\em globalization} of the right partial $S$-act $(A,\cdot)$ via the map $\iota$. 
\end{definition}

A partial action $\cdot$ can be {\em globalized} if it has a globalization. 
Lemma \ref{lem:strong} implies the following.

\begin{proposition}\label{prop:globstrong}
If a partial action of a semigroup can be globalized then it is strong. 
\end{proposition}

\begin{definition} ($A$-generated globalization) \label{def:a-gen}
Let the global action $*$ of $S$ on $B$ be a globalization of the partial action $\cdot$ of $S$ on $A$ via the map $\iota\colon A\to B$. We say that $B$ is $A$-{\em generated} if any $b\in B$ can be written as $b=\iota(a)*s$ for some $a\in A$ and $s\in S$.
\end{definition}

Let $(A,\cdot)$ be a partial $S$-act and $(B,*)$ a global $S$-act which is an $A$-generated globalization of $(A,\cdot)$ via the map $\iota\colon A\to B$. Since $\iota(A)\subseteq B$, for every $b\in A$ there are $a\in A$ and $s\in S$ such that  $\iota(b)=\iota(a)*s$, hence we have $\iota(b)=\iota(a\cdot s)$, so that $b=a\cdot s$. It follows that $(A,\cdot)$ is necessarily unitary. So $A$-generated globalizations do not exist for non-unitary $S$-acts. If $(A,\cdot)$ is unitary and $(B,*)$ is its globalization via the map $\iota$, then $\iota(A)$ is a subset of the set $C=\{\iota(a)*s\colon a\in A, s\in S\}$. The elements of $\iota(A)$ are those $\iota(a)*s\in C$, for which $a\cdot s$ is defined. The following statement is now immediate.

\begin{lemma}
Let the global $S$-act $(B,*)$ be a globalization of the unitary partial $S$-act $(A,\cdot)$ via the map $\iota\colon A\to B$. Let $C=\{\iota(a)*s\colon a\in A, s\in S\}$. Then $C$ is a subact of $B$ and is the unique $A$-generated globalization of $A$ which is a subact of $B$.
\end{lemma}

It is easy to see that if $(A,\cdot)$ is a unitary global $S$-act then any of its  $A$-generated globalizations is isomorphic to $(A,\cdot)$.

Let $(A,\cdot)$ be a partial $S$-act. We define the category ${\mathcal{G}}(A,S,\cdot)$ of {\em globalizations} of $(A,\cdot)$ as follows. Its objects are triples $(B,\iota, *)$ where  $(B, *)$ is a globalization of $(A,\cdot)$ and $\iota\colon A\to B$ is an injective map satisfying  Definition \ref{def:glob}. A morphism from $(B,\iota, *)$ to $(C,i, \circ)$ is a map $\varphi\colon B\to C$ which is a morphism of right $S$-acts, that is, $\varphi(b*s) = \varphi(b)\circ s$ holds for all $b\in B$ and $s\in S$, and which satisfies $\varphi\iota=i$. In the case where $(A,\cdot)$ is unitary let ${\mathcal{G}}_A(A,S,\cdot)$ be the full subcategory of ${\mathcal{G}}(A,S,\cdot)$ whose objects are $A$-{\em generated  globalizations} of $(A,\cdot)$. When $\iota$ and $*$ are clear from the context, we sometimes abbreviate $(B,\iota,*)$ by $B$.

The proof of the following is straightforward.

\begin{proposition} \label{prop:unique} Let $(A,\cdot)$ be a unitary partial $S$-act and $(B,\iota, *)$ and $(C,i, \circ)$ be objects of the category 
${\mathcal{G}}(A,S,\cdot)$. 
\begin{enumerate} 
    \item\label{i:u1} Suppose that $B$ is $A$-generated and $\varphi: (B,\iota,*) \to (C,i,\circ)$ is a morphism in ${\mathcal{G}}(A,S,\cdot)$.
    Then $\varphi(B)$ is the $A$-generated subact of $C$. 
    \item \label{i:u2}  If $B$ is $A$-generated then there is at most one morphism from $(B,\iota, *)$  to $(C,i, \circ)$.
    \item \label{i:u3} If $B$ and $C$ are $A$-generated then a morphism from $(B,\iota, *)$  to $(C,i, \circ)$ is necessarily surjective. 
\end{enumerate}
\end{proposition}

\section{The tensor product globalization}\label{s:tensor}
\subsection{The construction and the universal property}
Let $\cdot$ be a right partial action of a semigroup $S$ on a set $A$. For $(a,s)$ and $(b,t)$ in $A\times S$ we put $(a,s)\to (b,t)$ if there is $u\in S$ such that $s=ut$ and the element $a\cdot u$ is defined and equals $b$. We have $(a, ut) \to (a\cdot u, t)$. Let $\rho$ be the smallest equivalence relation on  $A\times S$ which contains $\to$. We put $A\otimes S = (A\times S)/\rho$. Elements of $A\otimes S$ will be denoted by $a\otimes s$ where $a\otimes s$ is the $\rho$-class of $(a,s)$. By the definition, we have $a\otimes s = b\otimes t$ if and only if $(a,s) = (b,t)$ or there exist $n\geq 1$ and elements $(a,s)=(a_0,s_0), (a_1,s_1), \dots, (a_n,s_n) = (b,t)$ such that for all $i=0,\dots, n-1$ we have that either $(a_i,s_i)\to (a_{i+1},s_{i+1})$ or $(a_{i+1}, s_{i+1})\to (a_i,s_i)$ in which case we say that $(b,t)$ {\em is obtained from} $(a,s)$ {\em in} $n$ {\em steps}. We call $A\otimes S$ the {\em tensor product} of the right partial $S$-act $A$ and the left global $S$-act $S$ where $S$ acts on itself by left multiplication. 

For $a\otimes s\in A\otimes S$ and $t\in S$ we put 
\begin{equation}\label{eq:act1}
    (a\otimes s)*t = a\otimes st.
\end{equation}

\begin{lemma} The map $*\colon (A\otimes S)\times S\to A\otimes S$ given in \eqref{eq:act1} defines a global action of $S$ on $A\otimes S$.
\end{lemma}

\begin{proof} Let us first verify that the map $*$ is well-defined. Suppose that $a\otimes s = a' \otimes s'$. It is enough to assume that $(a',s')$ is obtained from $(a,s)$ in one step and then apply induction. Suppose that $(a,s)\to (a',s')$. That is, there is $u\in S$ such that $s=us'$ and the element $a\cdot u$ is defined and equals $a'$. For $t\in S$ we have
$a\otimes st = a\otimes us't = a\cdot u \otimes s't = a'\otimes s't$. If $(a',s')\to (a,s)$, the equality $a\otimes st = a'\otimes s't$ follows by symmetry. Hence $*$ is well-defined. That $*$ defines an action is clear since $(a\otimes s) * uv = a\otimes suv 
=(a\otimes su)*v = ((a\otimes s)*u)*v$.
\end{proof}

The following definition is motivated by the notion of a firm global action, see Definition~\ref{def:firm_global} and the subsequent discussion.

\begin{definition}\label{def:firm_partial} (Firm partial action)
We say that a partial action $\cdot$ of $S$ on $A$ is {\em firm} provided that it is unitary (that is, it satisfies condition (U)) and also it satisfies the condition:
\begin{enumerate}
\item[(F)] whenever $a\cdot s$ and $b\cdot t$ are defined and $a\cdot s = b\cdot t$, we have $a\otimes s= b\otimes t$ in $A\otimes S$.
\end{enumerate}
\end{definition}

\begin{lemma}\label{lem:firm} If $S$ is a monoid then condition {\em (F)} holds.
\end{lemma}

\begin{proof} Let $a,b\in A$ and $s,t\in S$ be such that $a\cdot s$, $b\cdot t$ are defined and $a\cdot s=b\cdot t$. Then $a\otimes s= a\otimes s1 = a\cdot s\otimes 1 = b\cdot t\otimes 1 = b\otimes t$.
\end{proof}

Suppose that $\cdot$ is firm and let $a\in A$. Let $b\in A$ and $t\in S$ be such that $a=b\cdot t$. Then there is a well-defined map
\begin{equation}\label{eq:delta}
    \delta\colon A\to A\otimes S, \; a\mapsto b\otimes t. 
\end{equation}

\begin{lemma}\label{lem:inj_delta} 
Suppose that $\cdot$ is firm and strong.
The map $\delta$ given in \eqref{eq:delta} is injective.
\end{lemma}

\begin{proof}
Suppose that $a,b\in S$ are such that $\delta(a)=\delta(b)$. Since $\cdot$ is unitary we have $a=a'\cdot s$ and $b=b'\cdot t$ for some $a',b'\in A$ and $s,t\in S$. Then $a'\otimes s = b'\otimes t$ and we need to show that $a'\cdot s = b'\cdot t$.  
If $(a',s) = (b',t)$ then $a'=b'$ and $s=t$, so that $a=a'\cdot s = b'\cdot t =b$. Otherwise, there exists a sequence 
$$
(a',s)=(a_0,s_0), (a_1,s_1), \dots, (a_n,s_n) = (b',t)
$$
such that, for all $i=0,\dots, n-1$, we have that either $(a_i,s_i)\to (a_{i+1},s_{i+1})$ or $(a_{i+1}, s_{i+1})\to (a_i,s_i)$. 
For the first transition, we have two possibilities. 

1) If $(a',s)\to (a_1,s_1)$, there exists $u\in S$ such that $a'\cdot u$ is defined,  $s=us_1$ and $a'\cdot u=a_1$. Since $a'\cdot u, a'\cdot us_1$ are defined, strongness implies that $(a'\cdot u)\cdot s_1$ is defined and 
$a'\cdot s = a'\cdot us_1=(a'\cdot u)\cdot s_1 = a_1\cdot s_1$. 

2) If $(a_1,s_1)\to (a',s)$, there exists $u\in S$ such that $s_1=us$, $a_1\cdot u$ is defined and $a_1\cdot u = a'$. 
Since $a'\cdot s = (a_1\cdot u)\cdot s$ is defined, also $a_1\cdot us = a_1\cdot s_1$ is defined and $a'\cdot s = a_1\cdot s_1$. 

Continuing in this way we see that also $a_2\cdot s_2,\ldots, a_{n-1}\cdot s_{n-1}$ are defined and $a'\cdot s = a_1\cdot s_1 = \cdots =  a_n\cdot s_n = b'\cdot t$.
\end{proof}

We now state the main result in this section.

\begin{theorem} \label{th:tensor} Let $S$ be a semigroup and $(A,\cdot)$ a firm and strong partial $S$-act. Then:
\begin{enumerate}
    \item The global  $S$-act $(A\otimes S,*)$ is an $A$-generated globalization of  $(A,\cdot)$ via the map $\delta\colon A\to A\otimes S$ given in \eqref{eq:delta}.
    \item The globalization $(A\otimes S,*)$ has the following universal property: if  $(B,\circ)$ is a globalization of $(A, \cdot)$ via a map $\iota\colon A\to B$ then there is a unique morphism of right $S$-acts 
    $\varphi\colon A\otimes S\to B$ such that $\iota = \varphi\delta$, that is, the following diagram commutes:
   \begin{center}
\begin{tikzcd}
A \arrow[r, "\delta"] \arrow[rd, "\iota"] & [4mm] A\otimes S \arrow[d, "\varphi", dashed]\\[5mm]
&[6mm] B
 \end{tikzcd}
 \end{center}
\end{enumerate} 
\end{theorem}

\begin{proof}
(1) Note that $\delta(A) = \{a\otimes s\colon a\cdot s \text{ is defined}\}$. By Lemma \ref{lem:inj_delta}, the map $\delta$ is injective. Let $a\in A$. By (U), there are $b\in A$ and $s\in S$ such that $a=b\cdot s$. 
Due to (S), we have that $a\cdot t = (b\cdot s)\cdot t$ is defined if and only if $b\cdot st$ is defined which, in turn, is equivalent to  $b\otimes st \in \delta(A)$. If this is the case, we have $\delta (a\cdot t) = \delta((b\cdot s)\cdot t) = \delta(b\cdot st) = b\otimes st = (b\otimes s)*t = \delta(a)*t$, so that $*$ is a globalization of $\cdot$.
Since $a\otimes t = (b\cdot s)\otimes t = b\otimes st = (b\otimes s)*t = \delta(a)*t$, we see that $A\otimes S$ is $A$-generated.

(2) Let  $(B,\circ)$ be a globalization of $(A, \cdot)$ via  $\iota\colon A\to B$. For $a\otimes s\in A\otimes S$ we put $\varphi(a\otimes s) = \iota(a) \circ s$. To show that $\varphi$ is well defined,
assume that $a\otimes s = b\otimes t$ and show that $\iota(a)\circ s = \iota(b)\circ t$. Arguing by induction and due to symmetry, it suffices to assume that $(a,s)\to (b,t)$. Then there is $u\in S$ such that $s=ut$ and $a\cdot u$ is defined and equals $b$. We have
\begin{align*}
    \iota(a) \circ s & = \iota(a)\circ ut = (\iota(a)\circ u)\circ t & (\text{since } \circ \text{ is an action})\\
    & = \iota(a\cdot u)\circ t & (\text{by (G2)})\\
    & = \iota(b)\circ t, & (\text{since } a\cdot u=b)
\end{align*}
as desired. The map $\varphi$ is a morphism of $S$-acts because
\begin{align*}
    \varphi((a\otimes s)*t) & = \varphi(a\otimes st) & (\text{by the definition of } *)\\
    & =  \iota(a)\circ st & (\text{by the definition of } \varphi)\\
    & = (\iota(a)\circ s)\circ t & (\text{since } \circ \text{ is an action})\\
    & = \varphi(a\otimes s)\circ t. & (\text{by the definition of } \varphi)
\end{align*}
Further, let $a\in A$. Since $\cdot$ is unitary there are $b\in A$ and $s\in S$ such that the element $b\cdot s$ is defined and equals $a$. Then we have:
\begin{align*}
    \varphi(\delta(a)) & = \varphi(b\otimes s) & (\text{by the definition of } \delta)\\
    & =  \iota(b)\circ s & (\text{by the definition of } \varphi)\\
    & = \iota(b\cdot s) = \iota(a),& (\text{by (G2) and since } a=b\cdot s)\\
\end{align*}
so that $\varphi\delta=\iota$. Uniqueness of $\varphi$ follows from Proposition \ref{prop:unique}\eqref{i:u2}.
\end{proof}

\begin{corollary}\label{cor:initial} Let $\cdot$ be a firm and strong partial action of a semigroup $S$ on a set $A$. The triple $(A\otimes S, \delta, *)$ is an initial object in the category ${\mathcal{G}}(A,S,\cdot)$ and in the category ${\mathcal{G}}_A(A,S,\cdot)$. 
\end{corollary}

\begin{remark}\label{rem:mon} {\em Let $S$ be a monoid and $\cdot$ be its strong partial action in the sense of Hollings~\cite{H07} (that is, (PA), (Um) and (S) hold). By Lemma \ref{lem:monoid} $\cdot$ satisfies also (U). Now, from Lemma~\ref{lem:firm} we have that $\cdot$ is firm. It is easy to see that the globalization $(A\otimes S,*)$ of $(A,\cdot)$ from Theorem~\ref{th:tensor} coincides with that from \cite{MS04} and \cite{H07}.
Hence Theorem~\ref{th:tensor}(1) is a generalization of  \cite[Proposition 2.2]{MS04} (and of the construction in its proof). In addition, Corollary \ref{cor:initial} is a generalization of \cite[Theorem 5.10]{H07}.}
\end{remark}

\begin{remark}\label{rem:new}
{\em
Let $(A,\cdot)$ be a strong partial $S$-act and let $S^1$ be the monoid obtained from $S$ by attaching an external identity element $1$. Consider the partial $S^1$-act $(A,\circ)$ defined as follows. For all $a\in A$, $a\circ 1$ is defined and
equals $a$. If $s\neq 1$ then $a\circ s$ is defined if and only if $a\cdot s$ is defined in which case $a\circ s = a\cdot s$. Note that $(A,\circ)$ is firm since it is a partial monoid act and strong since so is $(A,\cdot)$.  Consider the tensor product $A\otimes S^1$ with the right action $*$ of $S^1$ on it given by \eqref{eq:act1}. By Theorem \ref{th:tensor}, it is an $A$-generated globalization of $(A,\circ)$. Since $1$ is an external identity element of $S^1$ the action  of $S^1$ on $A\otimes S^1$ restricts to the action of $S$ on $A\otimes S^1$, so that the global $S$-act $(A\otimes S^1,*)$ is a globalization of $(A,\cdot)$. It follows that any strong partial $S$-act can be globalized and thus the converse statement to that of Proposition \ref{prop:globstrong} holds true. }
\end{remark}

Let $(A,\cdot)$ be a strong partial $S$-act. It can be shown that there is a well-defined map $a\otimes s\mapsto a\otimes s$ from $A\otimes S$ to $A\otimes S^1$. Moreover, this map is surjective if and only if $(A,\cdot)$ is unitary, and injective if and only if $(A,\cdot)$ satisfies condition (F) of Definition \ref{def:firm_partial}. It follows that this map is bijective if and only if $(A,\cdot)$ is firm. In particular, for a firm and strong partial $S$-act $(A,\cdot)$ its globalizations $A\otimes S$ and $A\otimes S^1$ are isomorphic. We will not further develop this theme here.

\subsection{Firmness of the tensor product globalization}

\begin{definition}\label{def:firm_global} (Firm global actions)  A global action $*$ of a semigroup $S$ on a set $B$ is called {\em firm}, if the map  $\mu: B\otimes S\to B,  b\otimes s\mapsto b*s$, is bijective.
\end{definition}

The notion of a firm global action is a special case of that of a firm partial action, see Definition~\ref{def:firm_partial}.

Firm global acts are semigroup theoretic analogues of firm modules over rings, which probably appeared first in \cite[Definition~1.2]{Taylor}, but the term `firm' was introduced in \cite{Quillen}. They have been successfully used to develop Morita theory for semigroups without identity in \cite{Lawson2011} where the term `closed act' was used, in \cite{LMR18} where the term `firm act' first appeared and several subsequent papers. Somewhat earlier firm modules have been used in Morita theory of nonunital rings, for example in \cite{GM2000} and \cite{GM2001}. 

\begin{proposition}\label{prop:firm}
Let $S$ be factorizable (which means that $S=S^2$) and $A$ a firm and strong partial $S$-act. 
Then the global $S$-act $A\otimes S$ from Theorem \ref{th:tensor} is firm.
\end{proposition}

\begin{proof}
We need to show that the map $\mu\colon (A\otimes S)\otimes S\to A\otimes S$ given by $(a\otimes s)\otimes u\mapsto a\otimes su$ is bijective. From now on we write $a\otimes s\otimes u$ for $(a\otimes s)\otimes u$ using the fact that $(a\otimes s)\otimes u$ corresponds to $a\otimes (s\otimes u)$ under the canonical isomorphism $(A\otimes S) \otimes S \cong A\otimes (S \otimes S)$ which can be proved in a standard way. 
Let $a\otimes s \in A\otimes S$. Since $A$ is firm, it is unitary. Hence $a=b\cdot t$ for some $b\in A$ and $t\in S$. Then $a\otimes s = b\cdot t \otimes s = b\otimes ts = \mu(b\otimes t\otimes s)$. 
Thus $\mu$ is surjective. 

To show it is injective, assume that $a\otimes su = b\otimes tv$ in $A\otimes S$ and show that $a\otimes s \otimes u = b\otimes t\otimes v$ in $(A\otimes S)\otimes S$.
Suppose first that $(a,su) = (b,tv)$, that is, $a=b$ and $su = tv$.
Since $A$ is firm, condition (U) implies that there are $c\in A$ and $r\in S$ such that $c\cdot r$ is defined and equals $a$. Then
$a\otimes s\otimes u = c\cdot r \otimes s\otimes u = c\otimes rs\otimes u = c\otimes r \otimes su$ and similarly $a\otimes t \otimes v = c\otimes r \otimes tv$. It follows that $a\otimes s\otimes u = a\otimes t\otimes v$. 
Otherwise, $(b,tv)$ is obtained from $(a,su)$ in $n\geq 1$ steps. 

Arguing by induction and due to symmetry, it is enough to suppose that $c,d\in A$ and $x,y\in S$ are such that $(c,x)\to (d,y)$ and show that for any $p,q,s,t\in S$ such that $x=pq$ and $y=st$ we have $c\otimes p \otimes q = d\otimes s\otimes t$. (Such factorizations $x=pq$ and $y=st$ exist since $S$ is factorizable.) Since $A$ is firm, condition (U) implies that there are $c'\in A$ and $r\in S$ such that $c'\cdot r$ is defined and equals $c$. Since $(c,x)\to (d,y)$, there is $u\in S$ such that $x=uy=ust$, $c\cdot u$ is defined and $d=c\cdot u$. Applying the definition of $(A\otimes S)\otimes S$, we calculate:
\begin{multline*}
    c\otimes p\otimes q= c'\cdot r\otimes p\otimes q = c'\otimes rp \otimes q= c'\otimes r\otimes pq = c'\otimes r \otimes ust = \\
    c'\otimes rus \otimes t = (c'\cdot r)\cdot u \otimes s \otimes t = d\otimes s\otimes t, 
\end{multline*}
as needed.
\end{proof}

\subsection{The tensoring globalization functor and the reflection} 
In this subsection we extend the result of \cite[Theorem 1.1]{A03} (see also \cite{KhN18}) from partial actions of groups to firm and strong partial actions of semigroups. Let $\FSPACT(S)$ denote the category whose objects are firm and strong right partial $S$-acts and whose morphisms are morphisms between partial $S$-acts. Let $\FACT(S)$ be the full subcategory of $\FSPACT(S)$ whose objects are firm right global $S$-acts (the latter category has been considered in \cite{LMR18,Lawson2011}).

If $A$ is an object of the category $\FSPACT(S)$, by ${\mathrm{T}}(A)$ we denote the right global $S$-act $A\otimes S$. By Proposition \ref{prop:firm}, it is an object of the category $\FACT(S)$. Let $f\colon A\to B$ be a morphism in the category $\FSPACT(S)$. Define ${\mathrm{T}}(f)$ to be the map from $A\otimes S$ to $B\otimes S$ given by $a\otimes s \mapsto f(a) \otimes s$. It is easy to see that it is well defined and is a morphism of global $S$-acts. It is routine to check that the assignment $f\mapsto {\mathrm{T}}(f)$ is functorial, so that we have defined a functor ${\mathrm{T}} \colon \FSPACT(S) \to \FACT(S)$ which we call the {\em tensoring globalization functor.}

\begin{theorem}\label{th:reflective} Suppose that $S$ is factorizable.
The functor ${\mathrm{T}} \colon \FSPACT(S) \to \FACT(S)$ is a left adjoint to the inclusion functor ${\mathrm{I}}\colon \FACT(S)\to \FSPACT(S)$. Consequently, $\FACT(S)$ is a reflective subcategory of the category $\FSPACT(S)$ with the functor ${\mathrm{T}}$ being the reflector.
\end{theorem}  

\begin{proof}  Let $A$ be an object of the category $\FSPACT(S)$. Define the map $\eta_A\colon A\to {\mathrm{IT}}(A) = {\mathrm{I}}(A\otimes S)$, $a\mapsto b\otimes s$, where $a=b\cdot s$. Let $a\cdot t$ be defined. Since $\eta_A(a\cdot t) = a\otimes t = (b\cdot s)\otimes t$ and $\eta_A(a)*t = (b\otimes s)*t = b\otimes st = (b\cdot s)\otimes t$, $\eta_A$ is a morphism in the category $\FSPACT(S)$, and it is routine to see it is natural in $A$. 

Let $(B,\circ)$ be an object of the category $\FACT(S)$ and $f\colon A\to {\mathrm{I}}(B)$ a morphism in $\FSPACT(S)$. Define a map of global $S$-acts $g\colon {\mathrm{T}}(A) \to B$ by $g(a\otimes s) = f(a)\circ s$. Let us show that $g$ is well defined. Arguing by induction and applying symmetry, it is enough to assume that $(a,s)\to (b,t)$, that is, $s=ut$ and $a\cdot u$  is defined and equals $b$. Then 
$$f(a)\circ s = f(a)\circ ut = (f(a)\circ u)\circ t = f(a\cdot u)\circ t = f(b)\circ t,$$ as needed. Further, $g$ is a morphism in the category $\FACT(S)$ since $$g((a\otimes s)* t) = g(a\otimes st)= f(a)\circ st = (f(a) \circ s)\circ t = g(a\otimes s)\circ t.$$ In addition, for each $a\in A$, where $a=b\cdot s$, we have ${\mathrm{I}}(g)\eta_A(a) = g(b\otimes s) = f(b)\circ s = f(b\cdot s) = f(a)$, thus we have the following commuting triangle: 
\begin{center}
\begin{tikzcd}
A \arrow[d, "f"] \arrow[rd, "\eta_A"] \\[6mm]
{\mathrm{I}}(B) &[5mm] {\mathrm{IT}}(A)\arrow[l, "{\mathrm{I}}(g)", dashed]
\end{tikzcd}
\end{center}
The map $g$ is unique because if $h\colon {\mathrm{T}}(A) \to B$ is another morphism in the category $\FACT(S)$ satisfying ${\mathrm{I}}(h)\eta_A = f$ then for $a=b\cdot s \in A$ and $t\in S$ we have 
\begin{align*}
h(a\otimes t) & = h(b\cdot s\otimes t) = h(b\otimes st) & (\text{by the definition of } A\otimes S)\\
    & = h((b\otimes s)*t) = h(\eta_A(a)*t)  & (\text{by the definition of } * \text{ and } \eta_A)\\
    & = h(\eta_A(a))\circ t = f(a)\circ t  & (\text{since } h \text{ is a morphism and 
    } h\eta_A=f)\\
    & = g(a\otimes t), &  (\text{by the definition of } g)\\
\end{align*}
as desired. This completes the proof.
\end{proof}

\begin{remark}\upshape Semigroups $S$ and $T$ are called {\em Morita equivalent} if the categories $\FACT(S)$ and $\FACT(T)$ are equivalent (see \cite{Lawson2011}). In view of the connection between categories of global and partial acts demonstrated in Theorem~\ref{th:reflective}, it would be interesting to examine the equivalence relation on the class of all semigroups which is defined by requiring that the categories $\FSPACT(S)$ and $\FSPACT(T)$ are equivalent. Even in the case of monoids it is not known if this coincides with the Morita equivalence relation.
\end{remark}

\section{The ${\mathrm{Hom}}$-set globalization}\label{s:hom}

Let $(A,\cdot)$, $(B,\circ)$ be partial $S$-acts. Generalizing Definition \ref{def:morphism} we define a {\em partial morphism} from $(A,\cdot)$ to $(B,\circ)$ to be a partial map $\varphi\colon A\to B$ which respects the action,  that is, if $a\cdot s$, $\varphi(a)$ and $\varphi(a\cdot s)$ are defined then $\varphi(a)\circ s$ is defined and the equality $\varphi (a\cdot s) = \varphi(a)\circ s$ holds. The set of all partial morphisms from $(A,\cdot)$ to $(B,\circ)$ will be denoted by ${\mathrm{Hom}}_p(A,B)$.

Let $(A,\cdot)$ be a right partial $S$-act and consider $S$ as a right $S$-act under the action by right multiplication. Then ${\mathrm{Hom}}_p(S,A)$ 
is the set of all partial maps $f\colon S\to A$ such that if $s,st \in {\mathrm{dom}}(f)$ then $f(s)\cdot t$ is defined and $f(st)=f(s)\cdot t$.

For all $f\in {\mathrm{Hom}}_p(S,A)$ and $s\in S$ define the partial map $f*s\colon S\to A$ by
\begin{equation}\label{eq:dom1}
{\mathrm{dom}}(f*s) = \{t\in S\colon st\in {\mathrm{dom}}(f)\},
\end{equation}
\begin{equation}\label{eq:dom2}
(f*s)(t) = f(st) \text{ for all }  t\in {\mathrm{dom}}(f*s).
\end{equation}

\begin{lemma}
The assignment $*$ defines a right global action of $S$ on ${\mathrm{Hom}}_p(S,A)$.
\end{lemma}

\begin{proof}
First we check that $f*s \in {\mathrm{Hom}}_p(S,A)$. Suppose that $u,ut\in \dom(f*s)$. By \eqref{eq:dom1} this means that $su$, $sut \in \dom(f).$  It follows that $f(su)\cdot t$ is defined and equals $f(sut)$. Since, in addition, $(f*s)(u) =f(su)$ and $(f*s)(ut) =f(sut)$ by \eqref{eq:dom2}, we have that $(f*s)(u)\cdot t$ is defined and equals $(f*s)(ut)$, so $f*s \in {\mathrm{Hom}}_p(S,A)$.

Let $s,t\in S$ and show that $(f*s)*t = f*st$. Indeed,
$$
\dom((f*s)*t) = \{u\in S\colon tu\in \dom(f*s)\} = \{u\in S\colon stu\in \dom(f)\} = \dom(f*st)
$$
and for every $u\in \dom((f*s)*t)$ we have $((f*s)*t)(u) = (f*s)(tu) = f(stu) = (f*st)(u)$, as needed.
\end{proof}

For each $a\in A$ and $s\in S$ define the partial function $f_{a,s}\colon S\to A$ by
\begin{equation}\label{eq:dom3}
    {\mathrm{dom}}(f_{a,s}) = \{t\in S\colon a\cdot st \text{ is defined}\},
\end{equation}
\begin{equation}\label{eq:dom4}
f_{a,s}(t) = a\cdot st \text{ for all } t\in {\mathrm{dom}}(f_{s,a}).
\end{equation}
Let 
$$
A^S=\{f_{a,s}\colon a\in A, s\in S\}.
$$
\begin{lemma}  Suppose $\cdot$ is strong.
$A^S$ is a subact of ${\mathrm{Hom}}_p(S,A)$ and 
\begin{equation}\label{eq:act3}
f_{a,s}*t = f_{a,st}  \text{ for all } f_{a,s}\in A^S \text{ and } t\in S.
\end{equation}
\end{lemma}

\begin{proof} Let $f_{a,s}\in A^S$ and $u,v\in S$ be such that $f_{a,s}(u)$ and $f_{a,s}(uv)$ are defined. This means that $a\cdot su$ and $a\cdot suv$ are defined. Since $(A,\cdot)$ is strong,  $(a\cdot su)\cdot v=f_{a,s}(u)\cdot v$ is defined and equals $a\cdot suv$. Hence $f_{a,s}(uv) = a\cdot suv = (a\cdot su)\cdot v = f_{a,s}(u)\cdot v$. We have shown that $A^S\subseteq {\mathrm{Hom}}_p(S,A)$.
Observe that
\begin{align*}
    {\mathrm{dom}}(f_{a,s}*t) & = \{u\in S\colon tu\in {\mathrm{dom}}(f_{a,s})\} & (\text{by \eqref{eq:dom1}})\\
    & = \{u\in S\colon a\cdot stu \text{ is defined}\} & (\text{by \eqref{eq:dom3}})\\
    & = {\mathrm{dom}}(f_{a,st}). & (\text{by \eqref{eq:dom3}})
\end{align*}
In addition, if $u\in {\mathrm{dom}}(f_{a,st})$, we have
\begin{align*}
    (f_{a,s}*t)(u) & = f_{a,s}(tu) & (\text{by \eqref{eq:dom2}})\\
    & = a\cdot stu & (\text{by \eqref{eq:dom4}})\\
    & = f_{a,st}(u). & (\text{by \eqref{eq:dom4}})
\end{align*}
 We have shown that $f_{a,s}*t = f_{a,st}\in A^S$, which completes the proof.
\end{proof}

Suppose $\cdot$ is strong. For each $a\in A$ let $\lambda_a\colon S\to A$ be the map given by  ${\mathrm{dom}}(\lambda_a) = \{s\in S\colon a\cdot s \text{ is defined}\}$ and $\lambda_a(s) = a\cdot s$ for all $s\in {\mathrm{dom}}(\lambda_a)$. It is easy to see that $\lambda_a\in {\mathrm{Hom}}_p(S,A)$. We have defined the map $\lambda\colon A\to {\mathrm{Hom}}_p(S,A)$, $a\mapsto\lambda_a$.

\begin{lemma}\label{lem:as} Suppose that $\cdot$ is unitary and strong. For $a\in A$ let $b\in A$ and $s\in S$ be such that $a=b\cdot s$. Then $\lambda_a = f_{b,s}\in A^S$. Consequently, $$\lambda(A) = \{f_{b,s}\colon b\in A, s\in S \text{ and } b\cdot s \text{ is defined}\}.$$
\end{lemma}

\begin{proof}
We have ${\mathrm{dom}}(\lambda_a) = \{t\in S\colon a\cdot t \text{ is defined}\}=\{t\in S\colon (b\cdot s)\cdot t \text{ is defined}\}$. Also, ${\mathrm{dom}}(f_{b,s}) = \{t\colon b\cdot st \text{ is defined}\}$. Since $b\cdot s$ is defined and $\cdot$ is strong, we have ${\mathrm{dom}}(\lambda_a) = {\mathrm{dom}}(f_{b,s})$. In addition, for each $t\in {\mathrm{dom}}(\lambda_a)$ we have $\lambda_a(t) = a\cdot t = (b\cdot s)\cdot t = b\cdot st = f_{b,s}(t)$.
\end{proof}

The following notion is an extension to partial $S$-acts of the notion of a nonsingular global $S$-act. Unitary nonsingular global acts have been used in the Morita theory of semigroups starting from \cite{CS}. The term `nonsingular act' was introduced in \cite[Definition~3.1]{LM16}.

\begin{definition}
A  partial $S$-act $A$ is called {\em nonsingular}, if from $f_{a,s} = f_{b,t}$, where $a\cdot s$ is defined, it follows that $b\cdot t$ is defined and $a\cdot s = b\cdot t$.
\end{definition}

\begin{lemma}\label{lem:closed} If $S$ is monoid  then any partial $S$-act $(A,\cdot)$ is nonsingular.
\end{lemma}

\begin{proof}
Suppose that $s,t\in S$ and $a,b\in A$ are such that $f_{a,s} = f_{b,t}$ and that $a\cdot s$ is defined. Then $1\in {\mathrm{dom}}(f_{a,s})$. It follows that $1\in {\mathrm{dom}}(f_{b,t})$ which means that $b\cdot t$ is defined. Moreover, we have $a\cdot s = f_{a,s}(1) = f_{b,t}(1) = b\cdot t$. 
\end{proof}

Lemma \ref{lem:as} implies the following.

\begin{corollary}\label{cor:inj1}
Suppose that $\cdot$ is unitary, nonsingular and strong. Then the map $\lambda\colon A\to A^S$, $a\mapsto \lambda_a$, is injective.
\end{corollary}

We arrive at the main result of this section.

\begin{theorem} \label{th:hom} Let $S$ be a semigroup and $(A,\cdot)$ a unitary, nonsingular and strong partial $S$-act. Then:
\begin{enumerate}
    \item The global $S$-act $(A^S,*)$ is an $A$-generated globalization of $(A,\cdot)$ via the map $\lambda\colon A\to A^S$ given by $a \mapsto f_{b,s}$ where $a=b\cdot s$.
    \item The globalization $(A^S,*)$ has the following universal property: if $(B, \circ)$ is an $A$-generated globalization of $(A,\cdot)$ via a map $\iota\colon A\to B$ then there is a unique morphism of $S$-acts $\varphi\colon B\to A^S$ such that $\lambda = \varphi\iota$, that is, the following diagram commutes:
   \begin{center}
\begin{tikzcd}
A \arrow[r, "\lambda"] \arrow[rd, "\iota"] & [4mm] A^S \\[6mm]
&[6mm] B\arrow[u, "\varphi", dashed]
 \end{tikzcd}
 \end{center}
\end{enumerate} 
\end{theorem}

\begin{proof}
(1) We first show that  $(A^S,*)$ is a globalization of $(A,\cdot)$. By Corollary \ref{cor:inj1}, the map $\lambda$ is injective. Let $a=b\cdot t\in A$. For $s\in S$ we have that $a\cdot s$ is defined if and only if $(b\cdot t)\cdot s$ is defined. Due to (S), this is equivalent to $b\cdot ts$ being defined which implies $\lambda_a*s=f_{b,t}*s = f_{b,ts} = \lambda_{b\cdot ts} \in \lambda(A)$. 
Conversely, if $f_{b,ts}\in \lambda(A)$, nonsingularity implies that $b\cdot ts$ is defined and so $a\cdot s$ is defined by (S). If this is the case, we have $\lambda(a\cdot s) = \lambda(b\cdot ts) = f_{b,ts} = f_{b,t}*s = \lambda(a)*s$, as needed.

Let $f_{a,s}\in A^S$. Since $\cdot$ is unitary, there are $b\in A$ and $t\in S$ such that $b\cdot t$ is defined and equals $a$. Then $f_{a,s} = f_{b\cdot t,s} = f_{b, ts} = f_{b,t}*s = \lambda(a)*s$, where for the second equality we used (S), so that $A^S$ is $A$-generated.

(2)  Let $(B, \circ)$ be an $A$-generated globalization of $(A,\cdot)$ via  $\iota\colon A\to B$ and $b\in B$. Since $B$ is $A$-generated, there are $a\in A$ and $s\in S$ such that $b=\iota(a)\circ s$. We put $\varphi(\iota(a)\circ s) = f_{a,s}$.
Firstly, we show that $\varphi$ is well defined. Assume that $\iota(a)\circ s = \iota(b)\circ t$ and show that $f_{a,s} = f_{b,t}$. Let $x\in {\mathrm{dom}}(f_{a,s})$. This means that $a\cdot sx$ is defined. In view of (G1) and (G2), we have that $\iota(a)\circ sx \in \iota(A)$ and $\iota(a\cdot sx)= \iota(a)\circ sx$. Our assumption yields that $\iota(b)\circ tx = (\iota(b)\circ t)\circ x = (\iota(a)\circ s) \circ x = \iota(a)\circ sx$. Thus  $\iota(b)\circ tx \in \iota(A)$. In view of (G1), we obtain that $b\cdot tx$ is defined. It follows that  ${\mathrm{dom}}(f_{a,s}) \subseteq {\mathrm{dom}}(f_{b,t})$. By symmetry, the reverse inclusion also holds. Let $x\in {\mathrm{dom}}(f_{a,s})$. Then $f_{a,s}(x)= a\cdot sx$ and $f_{b,t}(x) = b\cdot tx$. Applying (G2), we have $\iota(a\cdot sx) = \iota(a)\circ sx = \iota(b) \circ tx = \iota (b\cdot tx)$. Since $\iota$ is injective, we conclude that $a\cdot sx = b\cdot tx$, so that $f_{a,s}(x) = f_{b,t}(x)$. This implies that $f_{a,s} = f_{b,t}$.

The map $\varphi$ is a morphism of $S$-acts because
\begin{align*}
    \varphi((\iota(a)\circ s)\circ t) & = \varphi(\iota(a)\circ st) &(\text{since } \circ \text{ is an action})\\
    & =  f_{a,st} & (\text{by the definition of } \varphi)\\
    & = f_{a,s}*t & (\text{by } \eqref{eq:act3})\\
    & = \varphi(\iota(a)\circ s) *t. & (\text{by the definition of } \varphi)
\end{align*}

Further, let $a\in A$. Since $\cdot$ is unitary there are $b\in A$ and $s\in S$ such that the element $b\cdot s$ is defined and equals $a$. Then 
$\varphi(\iota(a)) = \varphi(\iota(b)\circ s) = f_{b,s} = \lambda_a$, 
so that $\varphi\iota=\lambda$. Uniqueness of $\varphi$ follows from Proposition \ref{prop:unique}\eqref{i:u2}.
\end{proof}

\begin{remark}
{\em Let $(A,\cdot)$ be a strong partial $S$-act and let $S^1$ be the monoid obtained from $S$ by attaching an external identity element $1$. Consider the partial $S^1$-act $(A,\circ)$ defined in Remark \ref{rem:new}. Then it is strong, since so is $(A,\cdot)$, and also unitary and nonsingular by (Um) and Lemma~\ref{lem:closed}. It follows from Theorem \ref{th:hom} that the global $S^1$-act $(A^{S^1}, *)$  is a globalization of the partial $S^1$-act $(A,\circ)$. Since $1$ is an external identity element of $S^1$, the action of $S^1$ on $A^{S^1}$ restricts to that of $S$ on $A^{S^1}$ which globalizes the initial partial action $\cdot$ of $S$ on $A$. This provides another consturction, in addition to that in Remark \ref{rem:new}, of globalization of an arbitrary strong partial semigroup act.}
\end{remark}

Let $(A,\cdot)$ be a strong partial $S$-act. Let $C$ be the subset of $A^{S^1}$  consisting of all $f \in A^{S^1}$ which can be written as $f=f_{a,s}$ where $s\neq 1$. Then the assignment $f_{a,s}\mapsto f_{a,s}$ (where $a\in A$ and $s\in S$) is a well-defined surjective map from $C$ onto $A^S$. It can be shown that this map is injective if and only if $(A,\cdot)$ is nonsingular, and that its domain $C$ coincides with the whole $A^{S^1}$ if and only if $(A,\cdot)$ is unitary. Hence, it is an isomorphism between $A^{S^1}$ and $A^S$ if and only if $(A,\cdot)$ is both unitary and nonsingular. 

\begin{corollary}\label{cor:terminal}
Let $\cdot$ be a unitary, nonsingular and strong partial action of a semigroup $S$ on a set $A$. The triple $(A^S, \lambda, *)$ is a terminal object in the category  ${\mathcal{G}}_A(A,S,\cdot)$. 
\end{corollary}

A partial action $\cdot$ of $S$ on $A$ is called a {\em partially defined action}~\cite{K15} (or an {\em incomplete action}~\cite{GH09}) if it satisfies the condition that $a\cdot s$ and $(a\cdot s)\cdot t$ are defined if and only if $a\cdot st$ is defined (this is the case if and only if the corresponding premorphism $S\to {\mathcal{PT}}(A)$, where ${\mathcal{PT}}(A)$ is the partial transformation semigroup on $A$, is a homomorphism, see \cite{GH09, K15}). A partially defined action is necessarily strong. If $\cdot$ is a partially defined action, we say that $(A,\cdot)$ is a {\em partially defined} $S$-{\em act.}

Let $(A,\cdot)$ be a partially defined $S$-act and $c\not\in A$. It is known (see, e.g., \cite[p. 297]{H07} for the case where $S$ is a monoid) and easy to see that the assignment
 \begin{equation*}\label{eq:glob}
 a\circ s = \left\lbrace\begin{array}{ll} a\cdot s, & \text{ if } a\in A \text{ and } a\cdot s \text{ is defined,}\\
 c, & \text{ if } a\in A \text{ and } a\cdot s \text{ is not defined,}\\
 c, & \text{ if } a=c,
 \end{array} \right.
 \end{equation*}
defines a global $S$-act $(A\cup\{c\},\circ)$ which is a globalization of $(A, \cdot)$ via the map $\iota\colon A\to A\cup\{c\}$ which is identical on $A$. 

Let $0$  denote the only partial function $f\colon S\to A$ whose domain is the empty set.

\begin{proposition}\label{prop:isom1} Let $\cdot$ be a nonsingular and unitary partially defined action of $S$ on $A$ which is not a global action. 
Then $A^S= \lambda(A)\cup \{0\}$ and the universal globalization $(A^S, *)$ is isomorphic to $(A\cup \{c\}, \circ)$ via the map $f_{a,s} \mapsto a\cdot s$ if $f_{a,s}\in \lambda(A)$ and $0\mapsto c$.
\end{proposition}
\begin{proof}
Let $a\in A$ and $s\in S$ be such that $a\cdot s$ is not defined (they exist since $\cdot$ is not global).  Then, for any $t\in S$, $a\cdot st$ is not defined either, so that $t\not\in {\mathrm{dom}}(f_{a,s})$. Thus $f_{a,s}=0$. This and Lemma \ref{lem:as} imply the  equality $A^S= \lambda(A)\cup \{0\}$. The claim about the isomorphism is routine to verify.
\end{proof}

\section{Globalizations of a strong partial action of a monoid}\label{s:monoids}

In this section $S$ is a monoid. Combining Corollary \ref{cor:initial} and Corollary \ref{cor:terminal} we obtain the following.

\begin{theorem}\label{th:monoid}
Let $\cdot$ be a strong partial action of a monoid $S$ on a set $A$ in the sense of Hollings \cite{H07} (that is, conditions (PA), (Um) and (S) hold). Then $A\otimes S$ is an initial object and $A^S$ is a terminal object in the category ${\mathcal{G}}_A(A,S,\cdot)$. In particular, if $*$ is a global action of $S$ on a set $B$ which is an $A$-generated globalization of $\cdot$ via a map $\iota\colon A\to B$ then there are unique morphisms of global $S$-acts $A\otimes S\to B$, $a\otimes s\mapsto \iota(a)*s$, and $B\to A^S$, $\iota(a)*s\mapsto f_{a,s}$, such that the following diagram commutes:

\begin{center}
 \begin{tikzcd}
& [3mm] A \arrow[ld, "\delta" above=3pt] \arrow[d, "\iota"] \arrow[rd, "\lambda"]  & [4mm] \\[5mm]
A\otimes S \arrow[r] &[4mm] B\arrow[r] & [4mm] A^S
 \end{tikzcd}   
\end{center}
\end{theorem}

Theorem \ref{th:monoid} says that $A\otimes S$ is the `freest' possible $A$-generated globalization of $A$, and $A^S$ is the `smallest' possible such globalization.

\begin{remark} {\em The globalization $A\otimes S$ coincides with that considered by  Megrelishvili and Schr\"oder \cite{MS04} and Hollings \cite{H07} (although the tensor product notation is not used in \cite{MS04,H07}), see Remark \ref{rem:mon}.  The globalization $A^S$ is novel, as well as its property of being the terminal object of the category ${\mathcal{G}}_A(A,S,\cdot)$.}
\end{remark}

Let $G$ be a group and $A$ a set. A {\em partial action} of $G$ on $A$ \cite{Exel98, KL04} is a partial map $\cdot\colon A\times G\to A$, $(a,g)\mapsto a\cdot g$, which satisfies conditions (PA), (Um) and also the condition
\begin{enumerate}
    \item[(I)] If $a\cdot g$ is defined then $(a\cdot g)\cdot g^{-1}$ is defined and $(a\cdot g)\cdot g^{-1} = a$.
\end{enumerate}

It has been observed in \cite{MS04} that $\cdot$  is a partial action of $G$ on $A$ in the sense of \cite{Exel98, KL04} (that is, it satisfies (PA), (Um) and (I)) if and only if it satisfies (PA), (Um) and (S). The following result is a direct consequence of \cite[Proposition 3.3]{KL04} but we provide it with a proof, for completeness.

\begin{proposition}\label{prop:group} 
Let $S$ be a group and let $\cdot$ be a partial action of $S$ on a set $A$ in the sense of \cite{Exel98, KL04}. Let $*$ be a global action of $S$ on a set $B$ which is a globalization of $\cdot$. Then the map $\varphi\colon A\otimes S \to B$ from the proof of Theorem \ref{th:tensor}(2) is injective.
\end{proposition}

\begin{proof}
Let $\iota\colon A\to B$ be the injective map such that $(B,*)$ is a globalization of $(A,\cdot)$ via $\iota$. Recall that the map $\varphi\colon A\otimes S\to B$ from the proof of Theorem \ref{th:tensor}(2) is given by $a\otimes s\mapsto \iota(a)*s$. Assume that $\iota(a)*s=\iota(b)*t$. Acting on both sides of this equality by $s^{-1}$ from the right, we obtain $\iota(a) = \iota(b) * ts^{-1}$. Then $\iota(b)*ts^{-1} \in \iota(A)$, and (G1) implies that $b\cdot ts^{-1}$ is defined. In view of (G2), we have that $\iota(b\cdot ts^{-1})= \iota(b)*ts^{-1} = \iota(a)$. Injectivity of $\iota$ yields that $b\cdot ts^{-1} = a$. Therefore, $b\otimes t = b\otimes ts^{-1}s = b\cdot ts^{-1}\otimes s = a\otimes s$. This proves that $\varphi$ is injective.  
\end{proof}

\begin{corollary} \label{cor:groups} 
Let $S$ be a group and let $\cdot$ be a partial action of $S$ on a set $A$ in the sense of \cite{Exel98, KL04}. Let $(B, *)$ and $(C,\circ)$ be $A$-generated globalizations of  $(A,\cdot)$ via the injective maps $\iota$ and $i$, respectively. Then the map $B\to C$, $\iota(a)*s \mapsto i(a)\circ s$, is an isomorphism of 
global $S$-acts. In particular, $A\otimes S$ and $A^S$ are isomorphic via the map $A\otimes S\to A^S$ given by $a\otimes s\mapsto f_{a,s}$.
\end{corollary}

\begin{proof} The statement follows from Proposition \ref{prop:group} because,  since the globalization $B$ therein is $A$-generated, the map $\varphi$ is 
surjective (see Proposition \ref{prop:unique}\eqref{i:u3}). Hence it is a bijection, and  an isomorphism of 
global $S$-acts. 
\end{proof}

The following example shows that a partial action of a monoid can have infinitely many pairwise non-isomorphic globalizations and that $A\otimes S$ and $A^S$ can be far different. 

\begin{example}\label{ex:example}
{\em Let ${\mathbb N}$ be the set of positive integers and ${\mathbb N}^{0}=({\mathbb N}\cup\{0\},+)$ be the additive monoid of non-negative integers. 
Let the partial map ${\mathbb N}\times {\mathbb N}^{0}\to {\mathbb N}$, $(a,n)\mapsto a\cdot n$, be given by letting $a\cdot n$ be defined if and only if $a-n>0$ in which case $a\cdot n=a-n$. It is easy to see that $\cdot$ is a partially defined action of ${\mathbb N}^{0}$ on ${\mathbb N}$.

Let $B_{\mathcal{Z}}={\mathbb{Z}}$ be the set of integers and $\iota\colon {\mathbb N}\to B_{\mathcal{Z}}$ be the map which acts identically on ${\mathbb N}$.
For each $b\in B_{\mathcal{Z}}$ and $n\in {\mathbb N}^{0}$ put
$b*n = b-n$. Then $*$ is a global action of ${\mathbb N}^{0}$ on $B_{\mathcal{Z}}$ which is a  globalization of $\cdot$ via the map $\iota$. Since every $x\in B_{\mathcal{Z}}$ can be written as $x=a*n$ where $a\in {\mathbb N}$ and $n\in {\mathbb N}^{0}$, $B_{\mathcal{Z}}$ is ${\mathbb N}$-generated. In addition, for any $a,b\in {\mathbb N}$ and $n,m\in \mathbb{N}^0$ we have $a*n=b*m$ if and only if $a-n=b-m$. From this it easily follows that $B_{\mathcal{Z}}$ is isomorphic to ${\mathbb N}\otimes {\mathbb N}^{0}$ via the map $B_{\mathcal{Z}}\to {\mathbb N}\otimes {\mathbb N}^{0}$ given by $a*n \mapsto a\otimes n$.

For an integer $a\leq 0$ put $B_{a} = \{z\in {\mathbb Z}\colon z\geq a\}$. For each $b\in B_{a}$ and $n\in {\mathbb N}^{0}$ put
$$
b*_an = \left\lbrace\begin{array}{ll} b-n, & \text{if } b-n >a,\\
a, & \text{if } b-n\leq a.\end{array}\right.
$$
Then $*_a$ is an action of ${\mathbb N}^{0}$ on $B_a$ which globalizes $\cdot$ via the map $\iota$. In addition, $B_a$ is ${\mathbb N}$-generated.

If $a<b$ where $a,b\leq 0$ are integers, then the map $\varphi\colon B_a\to B_b$ given by  
$$
\varphi(x) = \left\lbrace\begin{array}{ll} x, & \text{if } x\geq b,\\
b, & \text{if } a\leq x \leq b-1,\end{array}\right.
$$
is a morphism of global ${\mathbb N}^{0}$-acts from $B_a$ to $B_b$. It can be checked that the map $\varphi$ can be alternatively given by $\varphi(c*_a n) = c*_b n$ where $c\in {\mathbb N}$ and $n\in \mathbb{N}^0$. 

Similarly, for an integer $b\leq 0$ the map $\psi\colon B_{\mathcal{Z}} \to B_b$ given by  
$$
\psi(x) = \left\lbrace\begin{array}{ll} x, & \text{if } x\geq b,\\
b, & \text{if } x \leq b-1,\end{array}\right.
$$
is a morphism of ${\mathbb N}^{0}$-acts from $B_{\mathcal{Z}}$ to $B_b$. It can be checked that the map $\psi$ can be alternatively given by $\psi(c* n) = c*_b n$ where $c\in {\mathbb N}$ and $n\in {\mathbb N}^{0}$. Proposition \ref{prop:isom1} implies that $B_0$ is isomorphic to ${\mathbb N}^{{\mathbb N}^{0}}$. If $a\neq b$ then $B_a$ is not isomorphic to $B_b$, since there is no bijection from $B_a$ to $B_b$ which would be identical on ${\mathbb N}$. Similarly, $B_{\mathcal{Z}}$ is not isomorphic to any $B_a$ where $a\leq 0$ is an integer.}
\end{example}

\section*{Acknowlegements} We are grateful to the referee for useful comments.


\begin{thebibliography}{99}
\bibitem{A03} F. Abadie, Enveloping actions and Takai duality for partial actions, {\em J. Funct. Anal.} {\bf 197}  (2003) (1), 14--67.

\bibitem{ABV19} M. M. S. Alves, E. Batista, J. Vercruysse, 
Dilations of partial representations of Hopf algebras, 
{\em J. Lond. Math. Soc.} {\bf 100} (2019) (1), 273--300. 

\bibitem{BEGRR22} A. Baraviera, R. Exel, D. Gon\c{c}alves, F. Rodrigues, D. Royer,  Entropy for partial actions of ${\mathbb{Z}}$, {\em Proc. Amer. Math. Soc.} {\bf 150} (2022) (3), 1089--1103. 

\bibitem{CS} Y.~Q. Chen, K.~P. Shum, 
Morita equivalence for factorisable semigroups,  
{\em Acta Math. Sin. (Engl. Ser.)} {\bf 17} (2001), 437--454.

\bibitem{CG12} C. Cornock, V. Gould, Proper two-sided restriction semigroups and partial actions, {\em J. Pure Appl. Algebra} {\bf 216} (2012), 935--949.

\bibitem{D19} M. Dokuchaev, Recent developments around partial actions, {\em S\~{a}o Paulo J. Math. Sci.} {\bf 13} (2019) (1), 195--247.

\bibitem{DKhK21} M. Dokuchaev, M. Khrypchenko, G. Kudryavtseva, Partial actions and proper extensions of two-sided restriction semigroups,  {\em J. Pure Appl. Algebra} {\bf 225} (2021) (9), paper no. 106649, 30 pp.

\bibitem{DKhS21} M. Dokuchaev, M. Khrypchenko, J.~J. Sim\'on,  Globalization of partial cohomology of groups, {\em Trans. Amer. Math. Soc.} {\bf 374} (2021) (3), 1863--1898.

\bibitem{Exel98} R. Exel, Partial actions of groups and actions of inverse semigroups, {\em Proc. Amer. Math. Soc.} {\bf 126} (1998), 3481--3494.

\bibitem{FMS21} G. Fonseca, G. Martini, L. Silva, Partial (co)actions of Taft and Nichols Hopf algebras on their base fields, {\em Internat. J. Algebra Comput.} {\bf 31} (2021) (7), 1471--1496.

\bibitem{GM2000} J.~L. Garc\'ia, L. Mar\'in, 
Some properties of tensor-idempotent rings, 
{\em Contemp. Math.} {\bf 259} (2000), 223--235.

\bibitem{GM2001} J.~L. García, L. Mar\'in, 
Morita theory for associative rings, 
{\em Comm. Algebra} {\bf 29} (2001), 5835--5856.

\bibitem{GH09} V. Gould, C. Hollings, Partial actions of inverse and weakly left ample semigroups, {\em J. Aust. Math. Soc.} {\bf 86} (2009), 355--377.

\bibitem{H07}  C. Hollings, 
Partial actions of monoids, {\em Semigroup Forum} {\bf 75} (2007), 293--316.

\bibitem{HV20} J. Hu, J. Vercruysse, 
Geometrically partial actions, {\em
Trans. Amer. Math. Soc.} {\bf 373} (2020) (6), 4085--4143.

\bibitem{KL04} J. Kellendonk, M.~V.~Lawson, Partial actions of groups,
{\it Internat. J. Algebra Comput.} {\bf
14} (2004) (1), 87--114.

\bibitem{Kh19} M. Khrypchenko, Partial actions and an embedding theorem for inverse semigroups, {\em Period. Math. Hungar.} {\bf 78} (2019) (1), 47--57.

\bibitem{KhN18} M. Khrypchenko, B. Novikov, Reflectors and globalizations of partial actions of groups, {\em J. Aust. Math. Soc.} {\bf 104} (2018) (3), 358--379. 


\bibitem{KKM} M. Kilp, U. Knauer, A.V. Mikhalev, 
{\em Monoids, acts and categories,} 
De Gruyter Expositions in Mathematics, 29. Walter de Gruyter \& Co., Berlin, 2000. xviii+529 pp.


\bibitem{K15} G. Kudryavtseva, 
Partial monoid actions and a class of restriction semigroups, {\em J. Algebra} {\bf 429} (2015), 342--370.

\bibitem{K19} G. Kudryavtseva, Two-sided expansions of monoids, {\em Internat. J. Algebra Comput.} {\bf 29} (2019) (8), 1467--1498.

\bibitem{LM16} V. Laan, L. M\'arki, Fair semigroups and Morita equivalence, {\em Semigroup Forum}, {\bf 92} (2016), 633--644.

\bibitem{LMR18} V. Laan, L. M\'arki, \"U. Reimaa,
Morita equivalence of semigroups revisited: firm semigroups, 
{\em J. Algebra} {\bf 505} (2018), 247--270. 

\bibitem{Lawson2011} M. V. Lawson, Morita equivalence of semigroups 
with local units, \emph{J. Pure Appl. Algebra} {\bf 215} (2011), 455--470. 

\bibitem{MS04} M. Megrelishvili,  L. Schr\"oder, Globalisation of confluent partial actions on topological and metric spaces, {\em Topology Appl.} {\bf 145} (2004), 119--145.

\bibitem{Quillen} D. G. Quillen, Module theory over nonunital rings, (1996), unpublished notes.

\bibitem{R22} J.~L.~V.~Rodr\'{i}guez, Partial actions on reductive Lie algebras, {\em Comm. Algebra} {\bf 50} (2022) (4), 1750--1767.

\bibitem{SV22} P. Saracco, J. Vercruysse,
Globalization for geometric partial comodules, {\em
J. Algebra} {\bf 602} (2022), 37--59.

\bibitem{Talwar96} S. Talwar, Strong Morita equivalence and a generalisation of the Rees theorem, {\em J. Algebra} {\bf 181} (1996) (2), 371--394

\bibitem{Taylor} J. L. Taylor, 
A bigger Brauer group, {\em Pacific J. Math.} {\bf 103} (1982), 163--203.

\end{thebibliography}
\end{document}